\documentclass[12pt,reqno]{amsart}
\usepackage{amsmath, amsfonts, amssymb, amsthm, hyperref}
\usepackage{bm}
\allowdisplaybreaks[4]
\def\l{\left}
\def\r{\right}
\def\bg{\bigg}
\def\({\bg(}
\def\){\bg)}
\def\t{\text}
\def\f{\frac}

\def\ls{\leqslant}
\def\gs{\geqslant}
\def\se {\subseteq}

\def\al{\alpha}

\def\eq{\equiv}

\def\la{\lambda}

\def\Z{\mathbb Z}

\def\N{\mathbb N}

\def\1{{\bf 1}}

\def\pmod #1{\ ({\rm{mod}}\ #1)}

\def\<{\langle}
\def\>{\rangle}
\theoremstyle{plain}
\newtheorem{theorem}{Theorem}[section]
\newtheorem{lemma}{Lemma}
\newtheorem{corollary}{Corollary}
\newtheorem{conjecture}{Conjecture}

\theoremstyle{definition}

\theoremstyle{remark}

\makeatletter
\@namedef{subjclassname@2020}{%
  \textup{2020} Mathematics Subject Classification}
\makeatother
 \vspace{4mm}

\begin{document}
\hbox{Preprint, {\tt arXiv:2407.04642}}
\medskip

\title[Some determinants involving binary forms]
{Some determinants involving binary forms}
\author{Yue-Feng She}
\address {(Yue-Feng She) Department of Applied Mathematics, Nanjing Forestry
University, Nanjing 210037, People's Republic of China}
\email{she.math@njfu.edu.cn}

\author{Zhi-Wei Sun}
\address {(Zhi-Wei Sun, corresponding author) Department of Mathematics, Nanjing
University, Nanjing 210093, People's Republic of China}
\email{zwsun@nju.edu.cn}

\keywords{Determinants, Legendre symbols, Jacobi symbols, Euler's totient function, polynomials over finite fields.
\newline \indent The second author is supported by the Natural Science Foundation of China (grant no. 12371004).}
\begin{abstract}
In this paper, we study arithmetic properties of certain determinants involving powers of $i^2+cij+dj^2$, where $c$ and $d$ are integers. For example,  for any odd integer $n>1$ with $(\frac dn)=-1$ we prove that $\det[(\frac{i^2+cij+dj^2}{n})]_{0\ls i,j\ls n-1}$ is divisible by $\varphi(n)^2$, where $(\frac{\cdot}{n})$ is the Jacobi symbol and $\varphi$ is Euler's totient function. This confirms a previous conjecture of Sun.
\end{abstract}
\maketitle

\section{Introduction}
\setcounter{lemma}{0}
\setcounter{theorem}{0}
\setcounter{equation}{0}
\setcounter{conjecture}{0}
\setcounter{remark}{0}
\setcounter{corollary}{0}

For each $n\times n$ matrix $M=[a_{ij}]_{1\ls i,j\ls n}$ over a commutative ring, we denote its determinant by $\det (M)$ or $\det [a_{ij}]_{1\ls i,j\ls n}$.
If $a_{ij}=0$ for all $1\ls i,j\ls n$ with $i\not=j$, then we simply write $M=[a_{ij}]_{1\ls i,j\ls n}$ as $\mathrm{diag}(a_{11},\ldots,a_{nn})$.
For various results over evaluations of determinants, one may consult the excellent survey papers \cite{K1,K2}. In this paper we study some determinants involving certain binary forms and related Jacobi symbols.

Let $a$ be any integer. For any odd prime $p$, the Legendre symbol $(\f ap)$ is given by
$$\l(\f ap\r)=\begin{cases}1&\t{if}\ p\nmid a\ \t{and}\ x^2\eq a\pmod p\ \t{for some}\ x\in\Z,
\\-1&\t{if}\ p\nmid a\ \t{and}\ x^2\eq a\pmod p\ \t{for no}\ x\in\Z,
\\0&\t{if}\ p\mid a.\end{cases}
$$
For any positive odd integer $n$, the Jacobi symbol $(\f an)$ is defined as follows:
$$\l(\f an\r)=\begin{cases}1&\t{if}\ n=1,
\\\prod_{i=1}^k(\f a{p_i})&\t{if}\ n=p_1\cdots p_k\ \t{for some primes}\ p_1,\ldots,p_k.
\end{cases}$$

Let $c,d\in\Z$. For any odd number $n>1$, Sun \cite{Sun19} introduced
$$
(c,d)_n:= \det \left[ \left(\f{i^2+cij+dj^2}{n}\right)\right]_{1\ls i,j\ls n-1}
$$
and
$$
[c,d]_n:= \det \left[ \left(\f{i^2+cij+dj^2}{n}\right)\right]_{0\ls i,j\ls n-1}.
$$
By \cite[Theorem 1.3]{Sun19}, $(c,d)_n=0$ if $(\f dn)=-1$, and $[c,d]_p$ is divisible by $p-1$
if $p$ is an odd prime with $(\f dp)=1$. For some results on $(c,d)_n$ and $[c,d]_n$ with $c$ and $d$ special, one may consult Krachun et al. \cite{KPSV}. For an odd prime $p$, the values of $(c,d)_p$ and $[c,d]_p$ are sometimes related to elliptic curves over the finite field $\mathbb F_p=\Z/p\Z$ (cf. \cite{KPSV,W}).

 In Section 2, we will prove the following result, which was first conjectured by Sun \cite[Conjecture 11.35]{Sconj}.

\begin{theorem}\label{cdp}
Let $c,d\in\Z$. For any odd number $n>1$ with $(\f{d}{n})=-1$, we have $\varphi(n)^2\mid [c,d]_n$, where $\varphi$ is Euler's totient function.
\end{theorem}

Let $c$ and $d$ be integers. By \cite[Theorem 1.2]{Sun24}, for any prime $p>3$ and $n\in\{(p+1)/2,\ldots,p-2\}$, we have
$$\det[(i^2+cij+dj^2)^n]_{0\ls i,j\ls p-1}\eq0\pmod p.$$
By \cite[Theorem 1.1]{WS}, for any odd prime $p$ with $(\f dp)=-1$ we have
$$\det[(i^2+cij+dj^2)^n]_{1\ls i,j\ls p-1}\eq0\pmod p$$
for all $n=1,\ldots,p-1$.

Suppose that $P(x,y)\in\Z[x,y]$ and its degree with respect to $x$ is smaller than $n\in\N$. For each $j=1,\ldots,n$, write
$$P(x,j)=\sum_{k=1}^n a_{jk}x^{k-1}$$
with $a_{j1},\ldots,a_{jn}\in\Z$. By \cite[Lemma 15]{K2}, we have
\begin{align*}\det[P(i,j)]_{1\ls i,j\ls n}&=\lim_{t\to0}\det[ta_{j1}(-i)^n+P(i,j)]_{1\ls i,j\ls n}
\\&=\lim_{t\to0}(1-n!t)\prod_{1\ls i<j\ls n}(j-i)\times\det[a_{jk}]_{1\ls j,k\ls n}
\\&=1!2!\cdots (n-1)!\times\det[a_{jk}]_{1\ls j,k\ls n}.
\end{align*}
In particular, if the degree of $P(x,y)$ with respect to $x$ is smaller than $n-1$, then $a_{1n}=\ldots=a_{nn}=0$ and hence
$$\det[P(i,j)]_{1\ls i,j\ls n}=1!2!\cdots (n-1)!\times\det[a_{jk}]_{1\ls j,k\ls n}=0.$$

We will establish the following result in Section 3.

\begin{theorem}\label{shift}
Let $p$ be an odd prime, and let
$$
H(X,Y)=\sum_{k=0}^{n}a_kX^kY^{n-k}
$$
with $a_0,\ldots,a_n\in\Z$.

{\rm (i)} If $n=p-1$, then
$$\det[x+H(i,j)]_{1\ls i,j\ls p-1}\equiv(x+a_0+a_{p-1})\prod_{k=1}^{p-2}a_k\pmod{p}.$$

{\rm (ii)} If $n=p-2$ or $p-1<n<2p-2$, then
$$
\det[x+H(i,j)]_{1\ls i,j\ls p-1}\equiv(-1)^n\prod_{k=0}^{p-2}\sum_{0\ls j\ls n\atop p-1\mid j-k}a_j\pmod{p}.
$$
\end{theorem}

By taking $H(X,Y)=(X^2+cXY+dY^2)^n$ with $c,d\in\Z$, we obtain the following result.

\begin{corollary}
Let $p>3$ be a prime, and let $c,d\in\Z$ and $n\in\{(p+1)/2,\ldots,p-2\}$.
Then  $\det[x+(i^2+cij+dj^2)^n]_{1\ls i,j\ls p-1}$ modulo $p$ is independent of $x$.
\end{corollary}

Let $p$ be an odd prime, and let $c,d\in\Z$. Sun \cite{Sun24} first introduced
$$D_p(c,d)=\det[(i^2+cij+dj^2)^{p-2}]_{1\ls i,j\ls p-1}$$
 motivated by his conjecture on $\det[1/(i^2-ij+j^2)]_{1\ls i,j\ls p-1}$ for $p\eq2\pmod3$ (cf.
 \cite[Remark 1.3]{Sun19}).
 For $(\f{D_p(1,1)}p)$ and $(\f{D_p(2,2)}p)$, one may consult
  \cite{LuoSun,WuSheNi}. See also \cite{SheWu} and \cite{LuoXia} for further results in this direction.

  Let $c,d\in\Z$. Sun \cite[Section 5]{SunProb} investigated
  $$\{c,d\}_n=\det\l[\l(\f{i^2+cij+dj^2}n\r)\r]_{1<i,j<n-1}$$
  with $n$ an odd number greater than $3$. Motivated by this, we
  study
$$
D_p^-(c,d):=\det [(i^2+cij+dj^2)^{p-2}]_{1<i,j<p-1}
$$
for any prime $p>3$. The difficulty of evaluating $D_p^-(c,d)$ lies in the fact that the indices do not run through a whole reduced system of residues modulo $p$.

For a prime $p$, let $\Z_p$ denote the ring of $p$-adic integers. It is well known that
each $p$-adic integer $\al$ can be written uniquely as a $p$-adic series
$\sum_{k=0}^\infty a_kp^k$ with $a_k\in\{0,\ldots,p-1\}$, which converges with respect to the $p$-adic norm $|\ |_p$. Hence we have the congruence
$\al\eq \sum_{k=0}^{n-1}a_kp^k\pmod{p^n}$ (in the ring $\Z_p$) for any positive integer $n$.
For example,
$$\f1{1-p}=\sum_{k=0}^\infty p^k\eq\sum_{k=0}^{n-1}p^k=\f{1-p^n}{1-p}\pmod{p^n}$$
for any positive integer $n$. A rational number is a $p$-adic integer if and only if
its denominator is not divisible by $p$. For $a,b,c\in\Z$ with $p\nmid b$, the congruence
$a/b\eq c\pmod p$ in the ring $\Z_p$ is actually equivalent to the congruence $a\eq bc\pmod p$
in the ring $\Z$. For instance, $2/3\eq 3\pmod 7$.

We will prove the following result in Section 4.

\begin{theorem}\label{formula}
Let $p>3$ be a prime, and let
$$
P(T)=a_0+a_1T+a_2T^2+\cdots +a_{p-2}T^{p-2},
$$
where $a_0,\ldots,a_{p-2}\in\Z_p$. Then we have
$$
\det \left[P(ij^{-1})\right]_{1<i,j<p-1}\equiv 4\sum_{i=0}^{(p-3)/{2}}\hat{a}_{2i}\times\sum_{i=0}^{(p-3)/{2}}\hat{a}_{2i+1}\pmod{p},
$$
where
$$
\hat{a}_k=\prod_{0\ls j\ls p-2 \atop 2\mid j-k, j\neq k}a_j\qquad\t{for all}\ \ k=0,\ldots,p-2.
$$
\end{theorem}

For an odd prime $p$ and a $p$-adic integer $\al$, we define $(\f{\al}p)$
as the Legendre symbol $(\f rp)$, where $r$ is the unique integer in $\{0,\ldots,p-1\}$
with $\al\eq r\pmod p$. If $\al=a/b$ with $a,b\in\Z$ and $p\nmid b$, then $(\f{\al}p)$
coincides with the Legendre symbol $(\f{ab}p)$.

As an application of Theorem \ref{formula}, we will prove the following result.

\begin{corollary}\label{Dp^-(1,1)}
Let $p>3$ be a prime.

 {\rm (i)} When $p\equiv 2\pmod{3}$, we have
 $$D_p^-(1,1)\equiv2^{(p-8)/{3}}3^4\pmod{p}\ \ \t{and}\ \ \left(\f{D_p^-(1,1)}{p}\right)=\left(\f{2}{p}\right).$$

{\rm (ii)} When $p\equiv 7\pmod{9}$, we have $D_p^-(1,1)\equiv0\pmod{p}$.

{\rm (iii)} When $p\equiv 1,4\pmod{9}$, we have $$\left(\f{D_p^-(1,1)}{p}\right)=\l(\f{\Sigma_1\Sigma_2}{p}\r),$$
 where
\begin{align*}
\Sigma_1=\sum_{k=1}^{(p-1)/{6}}\l(\f{1}{18k-13}-\f{1}{18k-2}\r)+\f{1}{6},
\end{align*}
and
\begin{align*}
\Sigma_2=\sum_{k=1}^{(p-1)/{6}}\l(\f{1}{18k-4}-\f{1}{18k-11}\r)+\f{1}{6}.
\end{align*}
\end{corollary}

{\it Example 1.1}. Let us illustrate Corollary \ref{Dp^-(1,1)}(iii) with $p=19$. It is easy to verify that
$$D_p^-(1,1)\eq -5\pmod{p},\ \Sigma_1\eq3\pmod{p}\ \t{and}\ \Sigma_2\eq-8\pmod{p}.$$
Thus
$$\l(\f{D_p^-(1,1)}p\r)=\l(\f{-5}{19}\r)=\l(\f{3\times(-8)}{19}\r)=\l(\f{\Sigma_1\Sigma_2}p\r).$$

The following conjecture of the second author might stimulate further research.

\begin{conjecture} Let $p>3$ be a prime.

{\rm (i)} We have $p\mid D_p^-(2,2)$ if $p\eq7\pmod8$.

{\rm (ii)} We have $p\mid D_p^-(3,3)$ if $p>5$ and $p\eq2\pmod3$.

{\rm (iii)} We have $p\mid D_p^-(3,1)$ if $p\eq3,7\pmod{20}$.
\end{conjecture}

We are going to prove Theorem \ref{cdp}, Theorem \ref{shift}, Theorem
\ref{formula} and Corollary \ref{Dp^-(1,1)} in Sections 2, 3, 4 and 5, respectively.

 For convenience,
for a matrix $M=[m_{ij}]_{0\ls i,j\ls n}$, we call the row $(m_{i0},\ldots,m_{in})$ with $0\ls i\ls n$
the {\it $i$-row} of $M$ which is actually the $(i+1)$-th row of $M$, and define the {\it $j$-column} with $0\ls j\ls n$ similarly. Such terms will be used in Sections 3 and 4.

\section{Proof of Theorem \ref{cdp}}
\setcounter{lemma}{0}
\setcounter{theorem}{0}
\setcounter{equation}{0}
\setcounter{conjecture}{0}
\setcounter{remark}{0}
\setcounter{corollary}{0}
\begin{lemma}\label{not-squarefree-case}
Suppose that $n>1$ is odd and not squarefree.
Then, for any $c,d\in\Z$ we have $[c,d]_n=0$.
\end{lemma}
\begin{proof} Write $n=p^{\alpha}m$, where $p$ is an odd prime and $\alpha,m\in\mathbb{Z}^{+}=\{1,2,3,\ldots\}$ such that $\alpha>1$ and $p\nmid m$. By the Chinese Remainder Theorem, there exists a number $k\in\{1,\ldots, n-1\}$ such that $m\mid k$ and $k\equiv p\pmod{p^{\alpha}}.$ For any $0\ls i\ls n-1$, we have
\begin{align*}
\left(\f{i^2+cik+dk^2}{n}\right)=\left(\f{i^2+cik+dk^2}{m}\right)\left(\f{i^2+cik+dk^2}{p}\right)^{\alpha}
=\left(\f{i^2}{m}\right)\left(\f{i^2}{p}\right)^{\alpha}
=\left(\f{i^2+ci0+d0^2}{n}\right).
\end{align*}
Therefore $[c,d]_n=0$.
\end{proof}

 We are now ready to prove Theorem \ref{cdp}.
\medskip

 \noindent{\it Proof of Theorem \ref{cdp}.} In light of Lemma \ref{not-squarefree-case},
 it suffices to assume that $n$ is squarefree. Let
 $$
 P^{+}(n):=\l\{p:\ \t{$p$ is a prime divisor of $n$ with $\l(\f{d}{p}\r)=1$}\r\}
 $$
and
 $$
 P^{-}(n)=\l\{p:\ \t{$p$ is a prime divisor of $n$ with $\l(\f{d}{p}\r)=-1$}\r\}.
 $$
 By the Chinese Remainder Theorem,
\begin{align*}
 &\sum_{0\ls i\ls n-1 \atop (i,n)=1}\left(\f{i^2+cij+dj^2}{n}\right)\\
=&\sum_{0\ls i\ls n-1 \atop (i,n)=1}\prod_{p\in P^{+}(n)\cup P^{-}(n)}\left(\f{i^2+cij+dj^2}{p}\right)\\
=&\prod_{p\in P^{+}(n)\cup P^{-}(n)}\sum_{1\ls x\ls p-1}\left(\f{x^2+cxj+dj^2}{p}\right)\\
=&\prod_{p\in P^{+}(n) \atop p\mid (c^2-4d)j}\(p-1-\(\f{j}{p}\)^2\) \times \prod_{p\in P^{+}(n) \atop p\nmid (c^2-4d)j}(-2)\times \prod_{p\in P^{-}(n) \atop p\mid j}(p-1)\times \prod_{p\in P^{-}(n) \atop p\nmid j}0
\end{align*}
with the aid of the fact that $(\f{d}{p})=-1$ implies $p\nmid (c^2-4d)$.

Let $Q=\prod_{p\in P^{-}(n)}p$, and define the function $f:P^{+}(n)\to \Z$ by
\begin{equation*}
f(p)=\begin{cases} p-2 &\mbox{if $p\nmid (c^2-4d)$,}\\-2 &\mbox{if $p\mid (c^2-4d)$.}\end{cases}
\end{equation*}
Then
\begin{equation*}
 \sum_{0\ls i\ls n-1 \atop (i,n)=1}\left(\f{i^2+cij+dj^2}{n}\right)=\begin{cases} 0&\mbox{if $Q\nmid j$,}\\
 \varphi(Q)\times\prod_{p\in P^{+}(n)\atop p\mid j}(p-1)\times\prod_{p\in P^{+}(n)\atop p\nmid j}f(p)&\mbox{if $Q\mid j$.}\end{cases}
\end{equation*}
For any subset $A$ of $P^{+}(n)$, define $p(A)=\prod_{p\in A}p$. Via similar arguments, we get
\begin{align*}
 &\sum_{0\ls i\ls n-1 \atop (i,n)=p(A)}\left(\f{i^2+cij+dj^2}{n}\right)\\
 =&\begin{cases} 0&\mbox{if $Q\nmid j$ or $(p(A),j)>1$,}\\
 \varphi(Q)\times\prod_{p\in P^{+}(n)\setminus A\atop p\mid j}(p-1)\times\prod_{p\in P^{+}(n)\setminus A\atop p\nmid j}f(p)&\mbox{if $Q\mid j$ and $(p(A),j)=1$.}\end{cases}
\end{align*}
Thus, when $Q\nmid j$ we have $$\sum_{0\ls i\ls n-1 \atop (i,n)=1}\left(\f{i^2+cij+dj^2}{n}\right)=\prod_{p\in A}f(p)\times \sum_{0\ls i\ls n-1 \atop (i,n)=p(A)}\left(\f{i^2+cij+dj^2}{n}\right)=0.$$
When $Q\mid j$, we have
\begin{align*}
 &\(\sum_{0\ls i\ls n-1 \atop (i,n)=1}\left(\f{i^2+cij+dj^2}{n}\right)\)^{-1}\prod_{p\in A}f(p)
 \times\sum_{0\ls i\ls n-1 \atop (i,n)=p(A)}\left(\f{i^2+cij+dj^2}{n}\right)
\\ =&\ \begin{cases} 0&\mbox{if $(p(A),j)>1$,}\\
                 1&\mbox{if $(p(A),j)=1$.}\end{cases}
\end{align*}

Let $\mu$ be  the M\"obius function. Then
\begin{align*}
 &\sum_{A\se P^{+}(n)}\mu(p(A))\prod_{p\in A}f(p)\times\sum_{0\ls i\ls n-1 \atop (i,n)=p(A)}\left(\f{i^2+cij+dj^2}{n}\right)\\
 =&\sum_{0\ls i\ls n-1 \atop (i,n)=1}\left(\f{i^2+cij+dj^2}{n}\right)\times\sum_{A\se P^{+}(n)\atop (p(A),j)=1}\mu(p(A))\\
 =&\sum_{0\ls i\ls n-1 \atop (i,n)=1}\left(\f{i^2+cij+dj^2}{n}\right)\times\sum_{d\mid \f{p(P^{+}(n))}{(p(P^{+}(n)),j)}}\mu(d)\\
 =&\begin{cases}\varphi(n)&\mbox{if $j=0$,}\\0&\mbox{otherwise.}\end{cases}
\end{align*}
The last equality follows from the well-known identity (cf. \cite[p.\,19]{IR})
\begin{equation*}
\sum_{d\mid k}\mu(d)=\begin{cases}1&\mbox{if $k=1$,}\\0&\mbox{if}\ k\in\{2,3,\ldots\}.\end{cases}
\end{equation*}
Thus, via certain elementary row transformations we obtain that $[c,d]_n=\det [a_{ij}]_{0\ls i,j\ls n-1}$, where
\begin{equation*}
  a_{ij}=\begin{cases}
           \varphi(n) & \mbox{if $i=1$\ \t{and}\ $j=0$, }  \\
           0 &\mbox{if $i=1$\ \t{and}\ $j\not=0$,}\\
           \left(\f{i^2+cij+dj^2}{n}\right) & \mbox{otherwise}.
         \end{cases}
\end{equation*}
Similarly,
\begin{align*}
 \sum_{A\se P^{+}(n)}\mu(p(A))\prod_{p\in A}f(p)\times\sum_{0\ls j\ls n-1 \atop (j,n)=p(A)}\left(\f{i^2+cij+dj^2}{n}\right)
 =\begin{cases}-\varphi(n)&\mbox{if $i=0$}\\0&\mbox{otherwise,}\end{cases}
\end{align*}
and hence $\det [a_{ij}]_{0\ls i,j\ls n-1}=\det [b_{ij}]_{0\ls i,j\ls n-1}$, where
\begin{equation*}
  b_{ij}=\begin{cases}
           \varphi(n) & \mbox{if $i=1$ and $j=0$,}  \\
           -\varphi(n) & \mbox{if $i=0$ and $j=1$,}  \\
           0 &\mbox{if $i=1$ and $j\not=0$, or $i\not=0$ and $j=1$,}\\
           \left(\f{i^2+cij+dj^2}{n}\right) & \mbox{otherwise}.
         \end{cases}
\end{equation*}
Therefore,
$$[c,d]_n=\det[a_{ij}]_{0\ls i,j\ls n-1}=\det [b_{ij}]_{0\ls i,j\ls n-1}\eq0\pmod{\varphi(n)^2}.$$
This concludes our proof. \qed

\section{Proof of Theorem \ref{shift}}
\setcounter{lemma}{0}
\setcounter{theorem}{0}
\setcounter{equation}{0}
\setcounter{conjecture}{0}
\setcounter{remark}{0}
\setcounter{corollary}{0}

We need the following well-known Weinstein-Aronszajn identity (cf. \cite{WAI}).

\begin{lemma}\label{WA}
	Suppose that $A$ and $B$ are matrices over the complex field of sizes $l\times m$ and $m\times l$, respectively. Then
	$$
	\la^m\det(\lambda I_l-AB)=\lambda^{l}\det(\lambda I_m-BA),
	$$
where $I_n$ denotes the identity matrix of order $n$.
\end{lemma}

\medskip

\noindent{\it Proof of Theorem \ref{shift}.} We set $A=[i^j]_{1\ls i\ls p-1\atop 0\ls j\ls n}$ and $C=[c_{i,j}]_{0\ls i,j\ls n}$ with
\begin{equation*}
  c_{i,j}=\begin{cases}
           x &\mbox{if $i=0$ and $j=0$, }  \\
           a_i & \mbox{if $i+j=n$,}  \\
           0 &\mbox{otherwise.}
         \end{cases}
\end{equation*}
By Lemma \ref{WA},
\begin{equation}\label{WArelation1}
\begin{aligned}
&\det(\lambda I_{p-1}-[x+H(i,j)]_{1\ls i,j\ls p-1})\\
=&\det(\lambda I_{p-1}-ACA^{T})\\=&\lambda^{p-n-2}\det(\lambda I_{n+1}-CA^{T}A)\\=&\lambda^{p-n-2}\det(\lambda I_{n+1}-C[s_{i+j}]_{0\ls i,j\ls n}),
\end{aligned}
\end{equation}
where $s_k=\sum_{i=0}^{p-1}i^k$.  According to \cite[p.\,235]{IR},
\begin{equation}\label{s1}
  \sum_{i=0}^{p-1}i^k\equiv\begin{cases}
           -1\pmod{p} & \mbox{if $p-1\nmid k$},  \\
           0\pmod{p} & \mbox{if $p-1\mid k$.}
         \end{cases}
\end{equation}
So,  when $n=p-2$ we have
\begin{align*}
\det[x+H(i,j)]_{1\ls i,j\ls p-1}&\equiv\det C\times \det[s_{i+j}]_{0\ls i,j\ls n}
\\&\eq(-1)^{(p-1)/2}\prod_{k=0}^n a_k\times(-1)^{(p-3)/2}
=-\prod_{k=0}^na_k\pmod{p}.
\end{align*}

Let $D=[d_{ij}]_{0\ls i,j\ls n}$ be the matrix $C[s_{i+j}]_{0\ls i,j\ls n}$.
\medskip

{\it Case} 1. $3(p-1)/2\ls n< 2p-2$.

In this case, we have
\begin{equation*}
  d_{ij}\equiv\begin{cases}
           -x\pmod{p} &\mbox{if $i=0$ and $j\in\{0,p-1\}$,}  \\
           -a_{0}\pmod{p} &\mbox{if $i=0$ and $j+n\in\{2p-2,3p-3\}$,}  \\
           -a_{i}\pmod{p} &\mbox{if $i\gs1$ and $i-j\equiv n\pmod{p-1}$,} \\
           0\pmod{p} &\mbox{otherwise.}
         \end{cases}
\end{equation*}
Hence $\lambda I_{n+1}-D$ is congruent to the matrix
$$
\addtocounter{MaxMatrixCols}{20}
\begin{bmatrix}
\lambda+x      &         &a_{0}   &        &      &      &x              &      &a_{0}   &      &  \\
               & \lambda &        &\ddots  &      &      &               &      &        &\ddots&  \\
               &         &\ddots  &        &\ddots&      &               &      &        &      &a_{2n-3p+3}\\
               &         &        &\ddots  &      &\ddots&               &      &        &      &  \\
a_{n-p+1}      &         &        &        &\ddots&      &\ddots         &      &        &      &  \\
               &\ddots   &        &        &      &\ddots&               &\ddots&        &      &  \\
               &         &\ddots  &        &      &      &\ddots         &      &\ddots  &      &  \\
               &         &        &\ddots  &      &      &               &\ddots&        &\ddots&  \\
               &         &        &        &\ddots&      &               &      &\ddots  &      &a_{2n-2p+2}\\
               &         &        &        &      &\ddots&               &      &        &\ddots&  \\
a_{n}          &         &        &        &      &      &a_{n}          &      &        &      &\lambda\\
\end{bmatrix}.
$$
 (whose entries $0$ are not indicated) modulo $p$.
Subtracting the $k$-column from the $(k+p-1)$-column for $0\ls k\ls n-p+1$, we find that the last matrix is transformed to the matrix
$$
\addtocounter{MaxMatrixCols}{20}
\begin{bmatrix}
\lambda+x      &         &a_{0}   &        &      &      &-\lambda&        &        &       &  \\
               & \lambda &        &\ddots  &      &      &        &\ddots  &        &       &  \\
               &         &\ddots  &        &\ddots&      &        &        &\ddots  &       &  \\
               &         &        &\ddots  &      &a_{n-p}&       &        &        &\ddots &  \\
a_{n-p+1}      &         &        &        &\ddots&      &0       &        &        &       &-\lambda  \\
               &\ddots   &        &        &      &\ddots&        &        &        &       &  \\
               &         &\ddots  &        &      &      &\ddots  &        &        &       &  \\
               &         &        &\ddots  &      &      &        &\ddots  &        &       &  \\
               &         &        &        &\ddots&      &        &        &\ddots  &       &  \\
               &         &        &        &      &a_{n-1}&       &        &        &\ddots &  \\
a_{n}          &         &        &        &      &      &0       &        &        &       &\lambda\\
\end{bmatrix}.
$$
Adding the $k$-row to the $(k-p+1)$-row for $p-1\ls k\ls n$, we see that the last matrix is transformed to
$$
\addtocounter{MaxMatrixCols}{20}
\begin{bmatrix}
\lambda+x               &       &a_{0}+a_{p-1}           &      &                      &      &        &      &   \\
                        &\lambda&                        &\ddots&                      &      &        &      &   \\
                        &       &\ddots                  &      &a_{n-p}+a_{n-1}       &      &        &      &    \\
a_{n-p+1}+a_{n}         &       &                        &\ddots&                      &0     &        &      &  \\
                        &\ddots &                        &      &\ddots                &      &        &      & \\
                        &       &\ddots                  &      &                      &\ddots&        &      & \\
                        &       &                        &\ddots&                      &      &\ddots  &      & \\
                        &       &                        &      &a_{n-1}               &      &        &\ddots&\\
a_{n}                   &       &                        &      &                      &0     &        &      &\lambda \\
\end{bmatrix}.
$$
Thus, by \eqref{WArelation1}, $\det(\lambda I_{p-1}-[x+H(i,j)]_{1\ls i,j\ls p-1})$ is congruent to
\begin{align*}
\det
\begin{bmatrix}
\lambda+x               &       &               &a_{0}+a_{p-1}&      & \\
                        &\lambda&               &                        &\ddots& \\
                        &       &\ddots         &                        &      &a_{n-p}+a_{n-1}\\
a_{n-p+1}+a_{n}         &       &               &\ddots                  &      & \\
                        &\ddots &               &                        &\ddots&      \\
                        &       &a_{p-2}&                        &      &\lambda\\
\end{bmatrix}
\end{align*}
modulo $p$. Taking $\lambda=0$ we obtain that
\begin{align*}
  \det[x+H(i,j)]_{1\ls i,j\ls p-1}
  \equiv(-1)^n\prod_{k=0}^{n-p+1}(a_{k}+a_{k+p-1})\times \prod_{n-p+1<k<p-1}a_{k}\pmod{p}.
\end{align*}

{\it Case} 2. $p-1< n< 3(p-1)/2$.

In this case, we have
\begin{equation*}
  d_{ij}\equiv\begin{cases}
           -x\pmod{p} &\mbox{if $i=0$ and $j\in\{0,p-1\}$,}  \\
           -a_{0}\pmod{p} &\mbox{if $i=0$ and $j=2p-2-n$,}  \\
           -a_{i}\pmod{p} &\mbox{if $i\gs1$ and $i-j\equiv n\pmod{p-1}$,} \\
           0\pmod{p} &\mbox{otherwise.}
         \end{cases}
\end{equation*}
Hence $\lambda I_{n+1}-D$ is congruent to the matrix
$$
\addtocounter{MaxMatrixCols}{20}
\begin{bmatrix}
\lambda+x&         &a_{0}   &      &      &x       &      &      & \\
         & \lambda &        &\ddots&      &        &      &      & \\
         &         &\ddots  &      &\ddots&        &      &      & \\
a_{n-p+1}&         &        &\ddots&      &\ddots  &      &      & \\
         &\ddots   &        &      &\ddots&        &\ddots&      & \\
         &         &\ddots  &      &      &\ddots  &      &\ddots& \\
         &         &        &\ddots&      &        &\ddots&      &a_{2n-2p+2}\\
         &         &        &      &\ddots&        &      &\ddots& \\
a_{n}    &         &        &      &      &a_{n}&      &      &\lambda\\
\end{bmatrix}.
$$
modulo $p$. Via some arguments similar to the discussion in Case 1, we obtain that
\begin{align*}
  \det[x+H(i,j)]_{1\ls i,j\ls p-1}
  \equiv(-1)^n\prod_{k=0}^{n-p+1}(a_{k}+a_{k+p-1})\times \prod_{k=n-p+2}^{p-2}a_{k}\pmod{p}.
\end{align*}

{\it Case} 3. $n=p-1$.

In this case, we have
\begin{equation*}
  d_{ij}\equiv\begin{cases}
           -x-a_0\pmod{p} &\mbox{if}\ i=0\ \text{and}\ j\in\{0,p-1\},  \\
           -a_{i}\pmod{p} &\mbox{if}\ i\gs1\ \t{and}\ i-j\equiv 0\pmod{p-1}, \\
           0\pmod{p} &\mbox{otherwise.}
         \end{cases}
\end{equation*}
In light of \eqref{WArelation1}, we have
$$
\det(\lambda I_{p-1}-[x+H(i,j)]_{1\ls i,j\ls p-1})\equiv (\lambda+x+a_0+a_{p-1})\prod_{k=1}^{p-2}(\lambda+a_k)\pmod{p}.
$$
Taking $\lambda=0$, we immediately obtain the desired result.
\medskip

In view of the above, we have completed our proof of Theorem \ref{shift}. \qed

\section{Proof of Theorem \ref{formula}}
\setcounter{lemma}{0}
\setcounter{theorem}{0}
\setcounter{equation}{0}
\setcounter{conjecture}{0}
\setcounter{remark}{0}
\setcounter{corollary}{0}
We shall use the following useful lemma (cf. \cite{Vr}).

\begin{lemma}[The Matrix-Determinant Lemma]\label{matrix-determinant-lemma}
Let $H$ be an $n\times n$ matrix over the complex field,
and let ${\bf u}$ and ${\bf v}$ be two $n$-dimensional column vectors
whose components are complex numbers. Then
$$
\det(H+{\bf uv}^T)=\det H+{\bf v}^T{\rm adj}(H){\bf u},
$$
where ${\rm adj}(H)$ is the adjugate matrix of $H$.
\end{lemma}

\noindent{\it Proof of Theorem \ref{formula}.} We set $A=[i^j]_{2\ls i\ls p-2 \atop 0\ls j\ls p-2}$ and $C=[c_{ij}]_{0\ls i,j\ls p-2}$ with
\begin{align*}
c_{ij}=\begin{cases}a_i &\mbox{if $i+j=p-2$,}\\0 &\mbox{if $i+j\neq p-2$.}\end{cases}
\end{align*}
We also define $s_k:=\sum_{i=2}^{p-2}i^k$ for $k=0,1,2,\ldots$. In view of \eqref{s1},
\begin{align}\label{s2}
s_k\equiv\begin{cases}-3\pmod{p}&\mbox{if $p-1\mid k$,}\\-2\pmod{p}&\mbox{if $2\mid k$ and $p-1\nmid k$,}\\0\pmod{p}&\mbox{if $2\nmid k$}.\end{cases}
\end{align}

By Wilson's theorem, we have
\begin{align*}
\det[P(ij^{-1})]_{1<i,j<p-1}\equiv\det[P(ij^{-1})j^{p-2}]_{1<i,j<p-1}\equiv\det(ACA^T)\pmod{p}.
\end{align*}
Hence it suffices to focus on the matrix $ACA^T$ from now on. Applying Lemma \ref{WA} and \eqref{s2}, we obtain
\begin{equation}\label{WArelation2}
\begin{aligned}
&\ \det(\lambda I_{p-3}-ACA^T)\\
=&\ \lambda^{-2}\det(\lambda I_{p-1}-CAA^T)\\
=&\ \lambda^{-2}\det(\lambda I_{p-1}-C[s_{i+j}]_{0\ls i,j\ls p-2})\\
\equiv&\ \lambda^{-2}\det(\lambda I_{p-1}-[d_{ij}]_{0\ls i,j\ls p-2})\pmod{p},
\end{aligned}
\end{equation}
where
\begin{align*}
d_{ij}=\begin{cases}-3a_i&\mbox{if $p-1\mid j-i-1$,}\\-2a_i&\mbox{if $2\mid j-i-1$ and $p-1\nmid j-i-1$,}\\0&\mbox{if $2\nmid j-i-1$.}\end{cases}
\end{align*}
Subtracting the $0$-column from the $2k$-column, and subtracting the $1$-column from the $(2k+1)$-column for $1\ls k\ls (p-3)/2$, we find that the matrix $\lambda I_{p-1}-[d_{ij}]_{0\ls i,j\ls p-2}$ is converted to
$$
\begin{bmatrix}
\lambda      &3a_0    &-\lambda&-a_0    &\cdots  &-\lambda&-a_0   \\
2a_1         &\lambda & a_1    &-\lambda&\cdots  &0       &-\lambda\\
0            &2a_2    &\lambda & a_2    &        &        &       \\
\vdots       &\vdots  &        &\ddots  &\ddots  &        &       \\
\vdots       &\vdots  &        &        &\ddots  &\ddots  &       \\
0            &2a_{p-3}&        &        &        &\lambda & a_{p-3} \\
3a_{p-2}     &0       &-a_{p-2}&0       &\cdots  &-a_{p-2}&\lambda\\
\end{bmatrix}.
$$
Subtracting the $2k$-column times $2$ from the $0$-column, and subtracting the $(2k+1)$-column times $2$ from the $1$-column for $1\ls k\ls (p-3)/2$, we see that the last matrix is transformed to
$$
\begin{bmatrix}
(p-2)\lambda &pa_0        &-\lambda&-a_0    &\cdots  &-\lambda&-a_0   \\
0            &(p-2)\lambda& a_1    &-\lambda&\cdots  &0       &-\lambda\\
(p-2)\lambda &0           &\lambda & a_2    &        &        &       \\
\vdots       &\vdots      &        &\ddots  &\ddots  &        &       \\
\vdots       &\vdots      &        &        &\ddots  &\ddots  &       \\
(p-2)\lambda &0           &        &        &        &\lambda & a_{p-3} \\
pa_{p-2}     &(p-2)\lambda&-a_{p-2}&0       &\cdots  &-a_{p-2}&\lambda\\
\end{bmatrix}.
$$
It follows from \eqref{WArelation2} that
\begin{align*}
\det(ACA^T)&\equiv4\det
\begin{bmatrix}
1     &      &        &-a_0&\cdots&        &-a_0   \\
      &1     &a_1     &    &      &        & \\
1     &      &        &a_2 &      &        &       \\
\vdots&\vdots&        &    &\ddots&        &       \\
\vdots&\vdots&        &    &      &\ddots  &       \\
1     &      &        &    &      &        & a_{p-3} \\
      &1     &-a_{p-2}&    &\cdots&-a_{p-2}& \\
\end{bmatrix}\\
&=4\det
\begin{bmatrix}
1     &-a_0  &\cdots&-a_0   &        &      &        &  \\
1     &a_2   &      &       &        &      &        &\\
\vdots&      &\ddots&       &        &      &        & \\
1     &      &      &a_{p-3}&        &      &        & \\
      &      &      &       &a_1     &      &        &1\\
      &      &      &       &        &\ddots&        &\vdots \\
      &      &      &       &        &      & a_{p-3}&1\\
      &      &      &       &-a_{p-2}&\cdots&-a_{p-2}&1 \\
\end{bmatrix}\pmod p.
\end{align*}
Let ${\bf 1}$ denote the $(p-3)/2$-dimensional column vector whose entries are all 1. By Lemma \ref{matrix-determinant-lemma},
\begin{align*}
&\ \det(ACA^T)\\
\equiv&\ 4\det
\begin{bmatrix}
1                  &                                                                    &   &        \\
{\bf1}             &\mathrm{diag}(a_2,\cdots,a_{p-3})+a_0{\bf1}{\bf1}^T&   &    \\
                   &                                                                    &\mathrm{diag}(a_1,\cdots,a_{p-4})+a_{p-2}{\bf1}{\bf1}^T&{\bf1}   \\
                  &                                                                        &   &1        \\
\end{bmatrix}\\
\equiv&\ 4\det(\mathrm{diag}(a_2,\cdots,a_{p-3})+a_0{\bf1}{\bf1}^T) \det(\mathrm{diag}(a_1,\cdots,a_{p-4})+a_{p-2}{\bf1}{\bf1}^T)\\
\equiv&\ 4(\hat{a}_0+{\bf1}^T \mathrm{diag}(\hat{a}_2,\cdots,\hat{a}_{p-3}){\bf1})(\hat{a}_{p-2}+{\bf1}^T \mathrm{diag}(\hat{a}_1,\cdots,\hat{a}_{p-4}){\bf1})\\
\equiv&\ 4\sum_{i=0}^{(p-3)/2}\hat{a}_{2i}\times\sum_{i=0}^{(p-3)/2}\hat{a}_{2i+1}\pmod{p}.
\end{align*}
This concludes our proof of Theorem \ref{formula}. \qed

\section{Deduce Corollary \ref{Dp^-(1,1)} from Theorem \ref{formula}}
\setcounter{lemma}{0}
\setcounter{theorem}{0}
\setcounter{equation}{0}
\setcounter{conjecture}{0}
\setcounter{remark}{0}
\setcounter{corollary}{0}

\noindent{\it Proof of Theorem \ref{formula}.} By Fermat's little theorem, there exists a polynomial
$$
P(T)=\sum_{k=0}^{p-2}a_kT^k\in\Z_p[T]
$$
such that
$$
(T^2+T+1)^{p-2}\equiv P(T)\pmod{p}
$$
for any $T\in\{1,2,\ldots,p-1\}$. When $p\equiv1\pmod{3}$, by \cite[Corollary 2.1]{LuoSun} we may take
\begin{align}\label{ak1}
a_k=\begin{cases}k+{5}/{3}&\mbox{if $k\equiv0\pmod{3}$,}\\
                 -k-{4}/{3}&\mbox{if $k\equiv1\pmod{3}$,}\\
                 -{1}/{3}&\mbox{if $k\equiv2\pmod{3}$.}\end{cases}
\end{align}
When $p\equiv2\pmod{3}$, by \cite[Lemma 2.1]{WuSheNi} we may take
\begin{align}\label{ak2}
a_k=\begin{cases}{1}/{3}&\mbox{if $k\equiv0,2\pmod{3}$,}\\-{2}/{3}&\mbox{if $k\equiv1\pmod{3}$.}\end{cases}
\end{align}

{\it Case} 1. $p\equiv2\pmod{3}$.

Combine Theorem \ref{formula} with \eqref{ak2}, we obtain that
\begin{align*}
D_p^-(1,1)&\equiv\det[P(ij^{-1})]_{1<i,j<p-1}\\
&\equiv4\prod_{k=0}^{p-2}a_k\times\sum_{k=0}^{(p-3)/{2}}\f{1}{a_{2k}}\times\sum_{k=0}^{(p-3)/{2}}\f{1}{a_{2k+1}}\\
&\equiv2^{(p-8)/3}3^4\pmod{p}.
\end{align*}

{\it Case} 2. $p\equiv7\pmod{9}$.

 Note that $(p-4)/3, (2p-5)/3\in\{0,1,\ldots,p-2\}$. Since $(p-4)/3\equiv1\pmod{3}$, by \eqref{ak1} we have $a_{(p-4)/{3}}=-p/3\equiv0\pmod{p}$. Similarly, $a_{(2p-5)/3}=2p/3\equiv0\pmod{p}$ since $(2p-5)/3\equiv0\pmod{3}$. Furthermore, both $(p-4)/3$ and $(2p-5)/3$ are odd and hence $\hat{a}_k\equiv0\pmod{p}$ when $2\nmid k$. It follows from Theorem \ref{formula} that $D_p^-(1,1)\equiv0\pmod{p}$.

{\it Case} 3. $p\equiv1,4\pmod{9}$.

Suppose that $a_k\equiv0\pmod{p}$ for some $k\in\{0,\ldots,p-2\}$. Then $k\equiv0,1\pmod{3}$.
 If $k\equiv0\pmod{3}$, then $p\mid 3k+5$ and $0\ls k\ls p-4$, hence $3k+5=p$ or $3k+5=2p$, which implies that $p\equiv5,7\not\eq1,4\pmod{9}$.
 If $k\equiv1\pmod{3}$, then $p\mid 3k+4$ and $1\ls k\ls p-3$, hence $3k+4=p$ or $3k+4=2p$, which implies that $p\equiv7,8\not\eq1,4\pmod{9}$.

By the last paragraph, $a_k\not\equiv0\pmod{p}$ for all $k\in\{0,\ldots,p-2\}$.
 It is easy to verify that $a_k\equiv a_{p-3-k}\pmod{p}$ for all $k=0,\ldots,p-3$. Hence we may derive from Theorem \ref{formula} and \eqref{ak1} that
\begin{align*}
\left(\f{D_p^-(1,1)}{p}\right)=\left(\f{a_{(p-3)/{2}}a_{p-2}\times\sum_{k=0}^{(p-3)/{2}}\f{1}{a_{2k}}\times
\sum_{k=0}^{(p-3)/{2}}\f{1}{a_{2k+1}}}{p}\right)=\left(\f{3\Sigma_13\Sigma_2}{p}\right)=\left(\f{\Sigma_1\Sigma_2}{p}\right).
\end{align*}
This completes the proof of Corollary \ref{Dp^-(1,1)}. \qed

\end{document}